\documentclass[10pt,twoside,reqno]{amsart}
 \usepackage[top=1.4in, bottom=1.4in, left=1in, right=1in]{geometry}
\usepackage{graphicx}
\usepackage{amsmath,latexsym}
\usepackage{amsfonts,amssymb}
\usepackage{amsmath}
\usepackage{amscd, amsthm}
\usepackage{stmaryrd}
\usepackage{fancyhdr}
\usepackage{commath,enumerate}
\usepackage{indentfirst, hyperref}

\usepackage{comment}

\usepackage{pgf,tikz}
\usepackage{mathrsfs}
\usetikzlibrary{arrows,calc,positioning}
\usepackage{float}
\numberwithin{figure}{section}

\numberwithin{equation}{section}
\newtheorem{theorem}{Theorem}[section]
\newtheorem{lemma}[theorem]{Lemma}
\newtheorem{proposition}[theorem]{Proposition}
\newtheorem{corollary}[theorem]{Corollary}
\newtheorem{definition}[theorem]{Definition}
\newtheorem{remark}[theorem]{Remark}

\newcommand{\Rm}[1]{
  \textup{\uppercase\expandafter{\romannumeral#1}}
}
\let\olddefinition\definition
\renewcommand{\definition}{\olddefinition\normalfont}
\let\oldremark\remark
\renewcommand{\remark}{\oldremark\normalfont}

\newcommand{\C}{\mathbb{C}}
\newcommand{\D}{\mathcal{D}}

\newcommand{\F}{\mathcal{F}}
\renewcommand{\H}{\mathcal{H}}

\newcommand{\N}{\mathbb{N}}
\renewcommand{\O}{\mathcal{O}}
\renewcommand{\P}{\mathcal{P}}
\newcommand{\R}{\mathbb{R}}
\newcommand{\Rc}{\mathcal{R}}
\newcommand{\T}{\mathbb{T}}
\newcommand{\Z}{\mathbb{Z}}
\newcommand{\Time}{T}

\def\vp{\varphi}
\def\ve{\varepsilon}
\def\px{\partial_x}
\def\pt{\partial_t}

\newcommand{\diff}{\,\mathrm{d}}

\def\bel{\begin{equation}\label}
\def\beq{\begin{equation}}
\def\eeq{\end{equation}}
\def\bega{\begin{array}}
\def\enda{\end{array}}

\author{John K. Hunter}
\address{Department of Mathematics, University of California at Davis}
\email{jkhunter@ucdavis.edu}
\thanks{JKH was supported by the NSF under grant number DMS-1616988}
\author{Jingyang Shu}
\address{Department of Mathematics, University of California at Davis}
\email{jyshu@ucdavis.edu}
\author{Qingtian Zhang}
\address{Department of Mathematics, University of California at Davis}
\email{qzhang@math.ucdavis.edu}
\title[Approximate SQG Front Equation]{Local Well-posedness of an Approximate Equation for SQG Fronts}
\date{\today}

\begin{document}

\begin{abstract}
We prove local well-posedness in the Sobolev spaces $\dot{H}^s(\T)$, with $s>7/2$, for an initial value problem for a nonlocal, cubically nonlinear, dispersive equation that provides an approximate description of the evolution of surface quasi-geostrophic (SQG) fronts with small slopes.
\end{abstract}

\maketitle

\section{Introduction}
In this paper, we prove the local well-posedness of the initial value problem
\begin{align}
\label{sqgivp}
\begin{split}
&\varphi_t + \frac 12 \partial_x \bigg\{\varphi^2 \log|\partial_x| \varphi_{xx}
- \varphi \log|\partial_x| (\varphi^2)_{xx} + \frac 13 \log|\partial_x| (\varphi^3)_{xx}\bigg\}
= 2 \log|\partial_x| \varphi_x,
\\
&\varphi(x,0) = \varphi_0(x),
\end{split}
\end{align}
where $\vp \colon \T\times \R \to \R$, with $\T = \R/2\pi\Z$, is spatially periodic with zero mean,
and $\log|\partial_x|$ is the Fourier multiplier operator with symbol $\log|\xi|$.
As explained further in Section~\ref{sec-sqg}, this initial value problem provides a cubically nonlinear approximation for the motion of
a surface quasi-geostrophic (SQG) front with small slope located at $y=\vp(x,t)$.
Previous well-posedness results for SQG fronts include \cite{CoCoGa,FeRo11,Gan,Rod05}.

Before stating our main result, we introduce some notation that is used throughout the paper. We write the
Fourier series of a function $f : \T\to \C$ with Fourier coefficients $\hat f(\xi)=(\F f)(\xi)$ as
\[
f(x)=\sum\limits_{\xi\in\Z} \hat f(\xi) e^{i\xi x},  \qquad \hat f(\xi)=\frac1{2\pi} \int_{\T}f(x) e^{-i\xi x}\diff{x}.
\]
We
denote the Hilbert space of zero-mean, periodic functions with square-integrable weak derivatives of the order $s\in \R$ by
\begin{align*}
\begin{split}
\dot{H}^s(\T) &= \left\{f \colon \T \to \R \mid \text{$\hat{f}(0) = 0$,  $\norm{f}_{\dot{H}^s} < \infty$}\right\},
\\
\norm{f}_{\dot{H}^s}
&=\left(\sum_{\xi \in \Z_*} \abs{\xi}^{2s} |\hat{f}(\xi)|^2\right)^{1/2},
\end{split}
\end{align*}
where $\Z_* = \Z \setminus \{0\}$ is the set of nonzero integers.
For $\sigma\in \N$, we denote by $W^{\sigma,\infty}(\T)$
the Sobolev space of functions $u : \T\to \R$ with $L^\infty$-derivatives of order less than or equal to $\sigma$ and norm
\[
\|u\|_{W^{\sigma,\infty}} =\sum\limits_{k=0}^\sigma \sup_{x\in \T} \left|\partial_x^k u(x)\right|.
\]
We consider only spatially periodic functions in this paper, and, when convenient, we omit the $\T$.

We denote by
\begin{equation}
L = \log|\px|,\qquad D = -i \px,\qquad |D|^s = |\px|^s
\label{defL}
\end{equation}
the Fourier multiplier operators with symbols $\lambda(\xi)$, $\xi$, $|\xi|^s$, respectively, where
\[
\lambda(\xi) = \begin{cases} \log|\xi| & \text{if $\xi \in \Z_*$},
\\
0 & \text{if $\xi =0$}.
\end{cases}
\]
Finally, we denote by $T_u$ the Weyl para-product operator with $u$, which is described in more detail in
Section~\ref{sec-weyl}.

Our main result is the following.
\begin{theorem}
\label{lwpthm}
Let $s > 7/2$. If $\varphi_0 \in \dot{H}^s(\T)$ satisfies $\|T_{\vp_{0x}}^2\|_{L^2\to L^2} \leq C$ for some  $0< C < 2$,
then there exists $\Time > 0$ depending only on $\|\vp_0\|_{\dot{H}^s}$ and $C$
such that the initial value problem \eqref{sqgivp} has a unique solution with $\varphi \in C([0, \Time]; \dot{H}^s(\T))$.
The solution map $U(t): \dot{H}^s(\T)\to C([0, \Time]; \dot{H}^s(\T))$, where $U(t) : \vp_0(x)\mapsto \vp(x,t)$, is continuous on
$\dot{H}^s(\T)$ and Lipschitz continuous on $\dot{H}^{r}(\T)$ for $0\le r<s-1$, meaning that if $\vp, \psi \in C([0, \Time]; \dot{H}^s(\T))$ are solutions, then there exists a constant $M>0$ depending on
$\|\vp\|_{C([0, \Time]; \dot{H}^s)}$, $\|\psi\|_{C([0, \Time]; \dot{H}^s)}$ such that
\begin{equation}\label{stability}
\|\vp(\cdot,t)-\psi(\cdot,t)\|_{H^{r}}\leq M\|\vp(\cdot,0)-\psi(\cdot,0)\|_{H^{r}} \qquad \text{for all $t\in [0,T]$}.
\end{equation}
\end{theorem}

The difficulty in proving this result is that straightforward $\dot{H}^s$-estimates for \eqref{sqgivp} do not close, due to a logarithmic loss of derivatives \cite{HSh17}. We can, however, get closed estimates for a weighted $\dot{H}^s$-energy defined  by
\begin{equation}
E^{(s)}(\vp)=\int_{\T} |D|^{s} \vp\cdot\left(2-T_{\vp_x}^2\right)^{2s+1}|D|^{s} \vp \diff{x}.
\label{weighted_energy}
\end{equation}
Here, the term $2$ in the para-product operator $(2-T_{\vp_x}^2)$ comes from the linear dispersive term $2\log|\partial_x|\vp_x$ in \eqref{sqgivp}, and it allows one to control the nonlinear contribution from $T_{\vp_x}^2$.

The evolution equation in \eqref{sqgivp} is invariant under the change of variables $(x,t)\mapsto(-x,-t)$, so the same
local existence result holds backward in time. We remark that one gets the continuation of a solution unless $\|\vp(t)\|_{\dot{H}^s}$ blows up or $\|\vp_x(t)\|_{L^\infty}$ increases sufficiently that $(2-T_{\vp_x}^2)$ is no longer positive definite.\footnote{Slopes $\vp_x$ with $\|T_{\vp_{x}}^2\|_{L^2\to L^2} =2$ are, however, outside the regime in which \eqref{sqgivp} is applicable to SQG fronts, since it is derived under the assumption that $|\vp_x| \ll 1$.} A similar proof of local well-posedness applies if the dispersive term in \eqref{sqgivp} has the opposite sign, in which case one replaces $(2-T_{\vp_x}^2)$
by $(2+T_{\vp_x}^2)$ in \eqref{weighted_energy}, and the solution can be continued so long as $\|\vp(t)\|_{\dot{H}^s}$ remains
finite.

An outline of this paper is as follows. In Section~\ref{sec-sqg}, we explain how \eqref{sqgivp} arises as a description of SQG fronts and compare it with equations for generalized SQG fronts. In Section \ref{sec-weyl}, we use Weyl para-differential calculus to derive some estimates for the action of $L$ and $|D|^s$ on products, and in Section \ref{sec-simeqn}, we carry out a Bony decomposition of \eqref{sqgivp}, given in Lemma~\ref{bony-sqg}. In Section \ref{sec-apriori}, we use this decomposition to prove an \emph{a priori} estimate in Proposition~\ref{apriori}, and in Section \ref{Wp}, we construct solutions by a Galerkin method.

\section{SQG fronts}
\label{sec-sqg}

The generalized SQG equation is a transport equation in two space dimensions for an active scalar
$\theta(x,y,t)$,
\begin{align}
\theta_t+u\cdot \nabla\theta=0,\qquad
u=\nabla^{\perp}(-\Delta)^{-\alpha/2}\theta.
\label{gsqg}
\end{align}
Here, $(-\Delta)^{-\alpha/2}$ is a fractional inverse Laplacian, $\nabla^\perp = (-\partial_y,\partial_x)$, and $0<\alpha\le 2$ is a parameter. If $\alpha = 1$, then \eqref{gsqg} is the SQG equation \cite{sqg}, and if $\alpha = 2$, then \eqref{gsqg} is the stream function-vorticity equation for two-dimensional, inviscid, incompressible fluid flows \cite{majda}.

Equation \eqref{gsqg} has piecewise-constant, front solutions of the form
\[
\theta(x,y, t)=\begin{cases} 1/2 & \text{if $y>\vp(x,t)$},\\
-1/2 & \text{if $y<\vp(x,t)$},
\end{cases}
\]
where we assume that the front $y=\vp(x,t)$ is a graph, and we normalize the jump in $\theta$ across the front to
one without loss of generality.

For spatially periodic fronts with $\vp(x+2\pi,t) = \vp(x,t)$, one finds that
\begin{equation}
\vp_t - \int_\T G_\alpha\left(x-x', \vp-\vp'\right) \left(\vp_x - \vp_{x'}'\right) \diff{x'} = 0,
\label{fullsqg}
\end{equation}
where $\vp=\vp(x,t)$, $\vp'=\vp(x',t)$, and $G_\alpha(x,y)$ is the Green's function of
$(-\Delta)^{\alpha/2}$ on the cylinder $\T\times \R$.  For $0<\alpha < 2$,
we have, up to a constant factor,
\[
G_\alpha(x,y) = \frac{1}{\left(x^2 + y^2\right)^{1-\alpha/2}} + \sum_{n\in \Z_*}\left[
\frac{1}{\left((x-2\pi n)^2 + y^2\right)^{1-\alpha/2}} -\frac{1}{\left(2\pi|n|\right)^{2-\alpha}}\right],
\]
and for $\alpha = 2$, we have
\[
G_2(x,y) = -\frac{1}{2\pi} \log\left| \sin\left(\frac{z}{2}\right)\right|,\qquad z = x+iy.
\]

Expansion of \eqref{fullsqg} up to terms that are linear and cubic
in the slope $\vp_x$ leads to the following approximate equation for generalized SQG fronts  \cite{HSh17}:
\begin{equation}
\vp_t+\frac12 a_\alpha \partial_x\left\{ \vp^2A_\alpha\vp-\vp A_\alpha(\vp^2)
+\frac13A_\alpha(\vp^3)\right\}+b_\alpha B_\alpha\vp_x=0.
\label{gapprox}
\end{equation}
Here, $a_\alpha$, $b_\alpha$ are constants depending on $\alpha$, and the multiplier operators $A_\alpha$, $B_\alpha$ are given by
\[
A_\alpha = \partial_x^2 B_\alpha,\qquad B_\alpha = \begin{cases} |\partial_x|^{1-\alpha} & \text{if $\alpha \ne 1$},
\\
\log|\partial_x| & \text{if $\alpha = 1$}.\end{cases}
\]
If $\alpha=1$, then \eqref{gapprox} is the approximate SQG equation in \eqref{sqgivp}.

In qualitative terms, \eqref{gapprox} consists of a nonlocal, cubically-nonlinear equation in conservation form with a linear dispersive term proportional to $B_\alpha \vp_x$.
If $1<\alpha \le 2$, then the dispersive term is of order less than one, and it is not smoothing, but the dispersionless equation is hyperbolic in nature. The initial value problem for both the dispersive and dispersionless equation (with $b_\alpha=0$) is then locally well-posed in $\dot{H}^s(\T)$ for $s > 9/2$ \cite{HSh17}. 

If $0<\alpha<1$, then the dispersionless equation appears to lose fractional derivatives and
not be well-posed in any Sobolev space. In this case, however,
the dispersive term is smoothing of order greater than one, and it is sufficient to control the nonlinear term.
The global well-posedness of the initial value problem on $\R$ for the fully nonlinear front equation
\eqref{fullsqg} with small initial data and $0<\alpha <1$ is proved in \cite{CGI17}.

The SQG equation with $\alpha=1$ is a borderline case.
The dispersive term $2 L\varphi_x$ on the right-hand side of \eqref{sqgivp}
has logarithmically greater order than first order, while the nonlinear flux
on the left-hand side of \eqref{sqgivp} depends on a logarithmic derivative of $\varphi$.
In fact, as shown in Lemma~\ref{bony-sqg}, there is a cancelation of derivatives in the flux, and
\begin{align*}
\varphi^2 L\varphi_{xx}
- \varphi L(\varphi^2)_{xx} + \frac 13L (\varphi^3)_{xx}
&= 2[L,\vp] \vp_x^2 + \left[[L,\vp],\vp\right]\vp_{xx}
=  2L \left(T_{\vp_x}^2\vp\right) +\Rc,
\end{align*}
where $\Rc$ is a lower order remainder term. As a result,  the dispersionless equation appears to lose derivatives at a logarithmic rate.
A weak local well-posedness result for both the dispersive and dispersionless initial value problem is proved in \cite{HSh17}, in which
$\vp(\cdot,t) \in \dot{H}^{\tau(t)}(\T)$ for some function $\tau(t) > 9/2$ that decreases sufficiently rapidly in time.
In this paper, we prove that the dispersive initial value problem \eqref{sqgivp} is locally well-posed in $\dot{H}^s(\T)$ for
any fixed $s > 7/2$.

\section{Weyl para-differential calculus}\label{sec-weyl}
In this section, we use the Weyl para-differential calculus to prove several lemmas for the operators $L$ and $|D|^s$
defined in \eqref{defL}. Further discussion of the Weyl calculus and para-products can be found in \cite{BCD11, chemin, Hor, Tay00}.

Let $\chi : \R \to \R$ be a smooth function supported in the interval $\{\xi\in \R: |\xi|\leq \ve\}$
and equal to $1$  on $\{\xi\in \R: |\xi|\leq \ve'\}$, where $\ve' = 3\ve/4$ and $0<\ve\ll 1$.
If $u, v \in \D'(\T)$ are distributions on $\T$, then
we define the Weyl para-product $T_u v \in \D'(\T)$ by
\[
\F (T_u v)(\xi)=\frac1{2\pi}\sum\limits_{\eta\in \Z_*} \chi\left(\frac{|\xi-\eta|}{|\xi+\eta|}\right)\hat u(\xi-\eta)\hat v(\eta),
\]
where we use the convention that  $\chi\left({|\xi-\eta|}/{|\xi+\eta|}\right) = 0$ if $\xi+\eta=0$.

The smoothness of the para-product is determined by the high-frequency factor $v$.
If $u \in L^\infty$ and $v\in \dot{H}^s$, then
\[
\|T_u v\|_{\dot{H}^s} \leq C \|u\|_{L^\infty} \|v\|_{\dot{H}^s}.
\]
Here, and below, we use $C$ to denote a generic positive constant.
In addition, if $u\in W^{\sigma,\infty}$ for $\sigma\in \N$ and $v\in \dot{H}^{s+\sigma}$, then
we can transfer derivatives from the low-frequency to the high-frequency factor to get
\begin{align*}
\|T_{D^{\sigma} u} v\|_{\dot{H}^{s}} &\leq C \|u\|_{L^\infty} \|v\|_{\dot{H}^{s+\sigma}}.
\end{align*}
When $u$ is a real-valued $L^\infty$-function, the Weyl para-product $T_u$ is a self-adjoint, bounded linear operator on $L^2$;
the self-adjointness of $T_u$ allows us to define the weighted energy \eqref{weighted_energy}.

Bony's decomposition of the product $u v$ is given by
\begin{equation}
uv=T_uv+T_vu+R(u,v).
\label{bony}
\end{equation}
This decomposition is well-defined if, for example, $u \in W^{\sigma,\infty}$ and $v\in \dot{H}^s$
 with $s+\sigma > 0$, and then
\[
\|R(u,v)\|_{\dot{H}^{s+\sigma}} \le C \|u\|_{W^{\sigma,\infty}} \|v\|_{\dot{H}^s}.
\]
We use the notation $\O(f)$ to denote a term satisfying 
\[
\|\O(f)\|_{\dot{H}^s}\leq C\|f\|_{\dot{H}^s}
\]
whenever there exists  $s\in \R$ such that $f\in \dot{H}^s$. We also use $O(f)$ to denote a term satisfying $|O(f)|\leq C|f|$ pointwise.

\begin{lemma}\label{lem-luv}
If $u, v \in L^2$, then
\[\begin{aligned}
 L(uv)&= T_v L u+T_{Dv} D^{-1}u-\frac12 T_{D^2v} D^{-2} u+\frac13 T_{D^3 v} D^{-3} u + \O(T_{D^4 v} D^{-4} u)\\
 & +T_u L v +T_{Du} D^{-1}v-\frac12 T_{D^2u} D^{-2} v+\frac13 T_{D^3 u} D^{-3} v + \O(T_{D^4 u} D^{-4} v)
 +LR(u,v),
\end{aligned}
\]
where the remainder terms satisfy
\begin{align*}
\begin{split}
&\|\O(T_{D^4 v} D^{-4} u)\|_{\dot{H}^s}\leq C\|T_{D^4 v} D^{-4} u\|_{\dot{H}^s},
\qquad
\|\O(T_{D^4 u} D^{-4} v)\|_{\dot{H}^s}\leq C\|T_{D^4 u} D^{-4} v\|_{\dot{H}^s}.
\end{split}
\end{align*}
Moreover, if $u, Lu \in W^{\sigma,\infty}$ for an integer $\sigma \ge 0$, and $v \in \dot{H}^{s}$ with $s+\sigma > 0$, then
\bel{remd3}
\|LR(u,v)\|_{\dot H^{s+\sigma}}\leq C(\|u\|_{W^{\sigma,\infty}} + \|Lu\|_{W^{\sigma,\infty}})\|v\|_{\dot H^s},
\eeq
for some constant $C > 0$.
\end{lemma}

\begin{proof} Using Bony's decomposition \eqref{bony}, we only need to compute $L T_u v$ and $L R(u,v)$.

{\bf 1.} We shall prove that
\bel{lem-luv1}
LT_u v=T_u Lv+T_{Du}D^{-1}v-\frac12T_{D^2u}D^{-2}v+\frac13T_{D^3u}D^{-3}v+\O(T_{D^4 u} D^{-4} v).
\eeq
Indeed, by the definition of Weyl para-product, we have for $\xi\ne 0$ that
\begin{align}
\begin{split}
\F (LT_u v)(\xi)&= \frac1{2\pi} \log|\xi| \sum\limits_{\eta\in\Z_*}\chi\left(\frac{|\xi-\eta|}{|\xi+\eta|}\right)\hat u(\xi-\eta)\hat v(\eta)
\\
&= \frac1{2\pi} \sum\limits_{\eta\in\Z_*}\log|\xi-\eta+\eta|\chi\left(\frac{|\xi-\eta|}{|\xi+\eta|}\right)\hat u(\xi-\eta)\hat v(\eta).
\end{split}
\label{FLT}
\end{align}
If $(\xi,\eta)$ belongs to the support of $\chi({|\xi-\eta|}/{|\xi+\eta|})$, then
\begin{equation}
\left|\frac{\xi-\eta}{\eta}\right| \leq \frac{2\ve}{1-\ve}.
\label{ineq1}
\end{equation}
To prove this claim, we use the fact that
\begin{equation}
|\xi-\eta|\leq \ve|\xi+\eta|
\label{temp_ineq}
\end{equation}
on the support of $\chi({|\xi-\eta|}/{|\xi+\eta|})$ and consider two cases.
\begin{itemize}
\item If $|\xi+\eta|\leq|\eta|$, then $|\xi-\eta|\leq\ve |\eta|$, so
\[
\left|\frac{\xi-\eta}{\eta}\right|\leq\ve<\frac{2\ve}{1-\ve}.
\]
\item If $|\xi+\eta|>|\eta|$, then $\xi\eta>0$, so $|\xi-\eta| = \left||\xi| - |\eta|\right|$, and can we rewrite
\eqref{temp_ineq} as
\[
\left||\xi| - |\eta|\right|\leq \ve(|\xi|+|\eta|),
\]
which implies that
\[
\left(\frac{1-\ve}{1+\ve} \right)|\xi| \le |\eta|\leq \left(\frac{1+\ve}{1-\ve}\right)|\xi|,
\]
and \eqref{ineq1} follows in this case also.
\end{itemize}

Using the Taylor expansion
\[
\begin{aligned}
\log|\xi-\eta+\eta|&=\log|\eta|+\log\left|1+\frac{\xi-\eta}{\eta}\right|\\
&= \log|\eta|+\frac{\xi-\eta}{\eta}-\frac12\frac{(\xi-\eta)^2}{\eta^2}+\frac13\frac{(\xi-\eta)^3}{\eta^3}
+O\bigg(\frac{|\xi-\eta|^4}{\eta^4}\bigg)
\end{aligned}
\]
in \eqref{FLT}, we get that
\[
\begin{aligned}
\F & (L T_u v)(\xi) \\
&= \frac1{2\pi} \sum\limits_{\eta\in\Z_*} \bigg[\log|\eta|+\frac{\xi-\eta}{\eta}-\frac12\frac{(\xi-\eta)^2}{\eta^2}+\frac13\frac{(\xi-\eta)^3}{\eta^3}+O\bigg(\frac{|\xi-\eta|^4}{\eta^4}\bigg)\bigg] \chi\left(\frac{|\xi-\eta|}{|\xi+\eta|}\right)\hat u(\xi-\eta)\hat v(\eta)\\
&= \F \left[T_u L v+T_{Du}D^{-1}v-\frac12T_{D^2u}D^{-2}v+\frac13T_{D^3u}D^{-3}v+\O(T_{D^4 u} D^{-4} v)\right](\xi),
\end{aligned}
\]
which proves \eqref{lem-luv1}.

\begin{figure}
\centering
\includegraphics[scale=0.5]{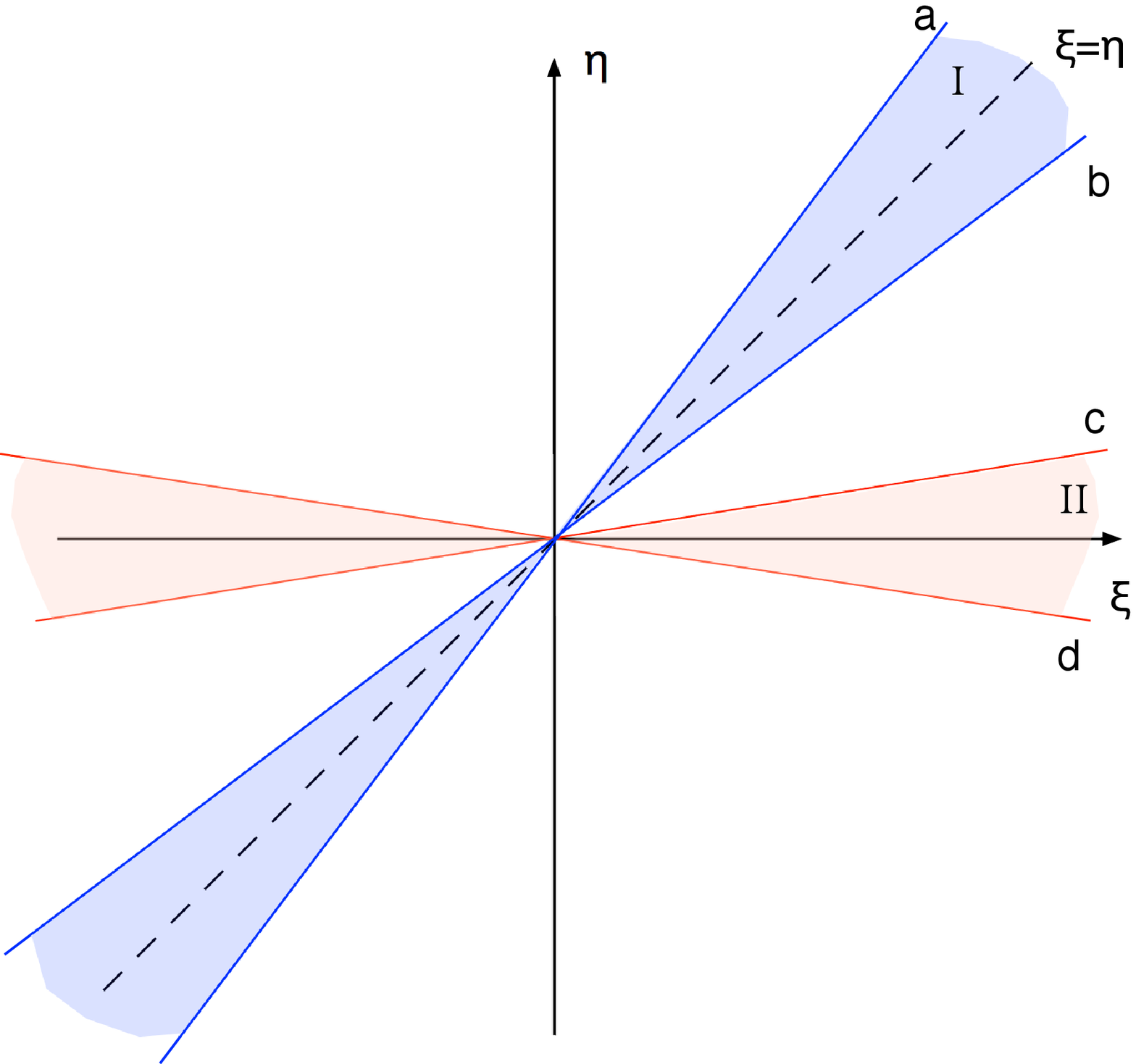}
\caption{\small Line a: $\eta=\frac{1+\ve'}{1-\ve'}\xi$. Line b: $\eta=\frac{1-\ve'}{1+\ve'}\xi$. Line c: $\eta=\frac{2\ve'}{1+\ve'}\xi$. Line d: $\eta=\frac{-2\ve'}{1-\ve'}\xi$. We have $\chi\left(\frac{|\xi-\eta|}{|\xi+\eta|}\right)=1$ in Region I, and $\chi\left(\frac{|\eta|}{|2\xi-\eta|}\right)=1$ in Region II. The support of $\rho(\xi,\eta)$ in \eqref{defrho} is contained
in the white region, where $\left|\frac{\xi-\eta}{\eta}\right|$ is bounded from above and away from zero. }
\label{fig-lemLuv}
\end{figure}

{\bf 2.} Next, we consider the remainder term $L R(u,v)$. For $\xi\ne 0$, we have
\begin{align}
\begin{split}
\F [LR(u,v)](\xi)&=\F [L(uv)-L(T_uv)-L(T_vu)](\xi)
\\
&= \frac1{2\pi} \log|\xi| \sum\limits_{\eta\in \Z_*} \rho(\xi,\eta)\hat u(\xi-\eta)\hat v(\eta),
\\
\rho(\xi,\eta) &= 1-\chi\left(\frac{|\xi-\eta|}{|\xi+\eta|}\right)-\chi\left(\frac{|\eta|}{|2\xi-\eta|}\right).
\end{split}
\label{defrho}
\end{align}
As illustrated in Figure \ref{fig-lemLuv}, there exist positive numbers $m, M > 0$ such that
\[
m\left|\xi-\eta\right|\leq|\eta|\leq M|\xi-\eta|
\]
for all $(\xi,\eta)$ in the support of $\rho(\xi,\eta)$, in which case
\begin{align*}
\log |\xi|&=\log |\xi-\eta+\eta|\leq \log [(1+M)|\xi-\eta|]=\log(1+M)+\log|\xi-\eta|,
\end{align*}
and $|\xi|^s\leq C|\eta|^s$. It follows that
\[
\|LR(u,v)\|_{\dot{H}^s}\leq C(\|u\|_{L^\infty} + \| Lu\|_{L^\infty})\|v\|_{\dot{H}^s}.
\]
Moreover, since $|\xi-\eta|$ and $|\eta|$ are comparable on the support of $\rho$, the remainder term also satisfies \eqref{remd3}
for any $\sigma \in \N$, which proves the Lemma.
\end{proof}

Setting $u=v$ in Lemma \ref{lem-luv}, we have the following corollary for $Lu^2$, which is of independent interest.
\begin{corollary}
If $u \in L^\infty \cap \dot{H}^s$ with $L u \in L^\infty$ and $s \ge 0$,
then there exists a constant $C > 0$ such that
\[
\|Lu^2-2uLu\|_{\dot{H}^s} \leq C (\|u\|_{L^\infty}+\|Lu\|_{L^\infty})\|u\|_{\dot{H}^s}.
\]
\end{corollary}
\begin{proof}
By Lemma \ref{lem-luv}, we have that
\[
L(u^2)=2T_u Lu+2T_{Du} D^{-1}u +\O(T_{D^2u}D^{-2}u)
\]
and
\[
2uLu=2T_u Lu+2T_{Lu} u+2R(Lu, u).
\]
Taking the difference of above two equations yields
\[
\begin{aligned}
\|Lu^2-2uLu\|_{\dot{H}^s}&= \|2T_{Du} D^{-1}u +\O(T_{D^2u}D^{-2}u)-2T_{Lu} u-2R(Lu, u)\|_{\dot{H}^s}\\
&\leq C (\|u\|_{L^\infty}+\|Lu\|_{L^\infty})\|u\|_{\dot{H}^s}.
\end{aligned}
\]
\end{proof}

The following lemma gives an expansion of $L(uvw)$ and an estimates of the remainder terms.

\begin{lemma}\label{lem-luvw}
If $u, v, w \in W^{3,\infty} \cap \dot{H}^s$, with $s \geq 0$, then
\[
\begin{aligned}
 L(uvw)=\sum\limits_{u,v,w}T_vT_w Lu+ (T_{Dv} T_w +T_v T_{Dw})D^{-1} u-\frac12[T_{D^2v} T_w+T_{D^2w} T_v&+2T_{Dv}T_{Dw}]D^{-2}u\\
&\quad +\mbox{remainder},
\end{aligned}
\]
where the summation is cyclic over $u,v,w$, and the remainder terms satisfy
\[
\|\mbox{remainder}\|_{\dot{H}^{s+2}}\leq C\left(\|u\|_{W^{3,\infty}}+\|v\|_{W^{3,\infty}}+\|w\|_{W^{3,\infty}}\right)^2(\|u\|_{\dot{H}^s}+\|v\|_{\dot{H}^s}
+\|w\|_{\dot{H}^s}),
\]
for some constant $C > 0$.
\end{lemma}

\begin{proof}
By Lemma \ref{lem-luv},
\bel{lem-luvw1}
\begin{aligned}
L[u(vw)]=& T_{vw}Lu+T_{D(vw)}D^{-1} u-\frac12T_{D^2(vw)} D^{-2}u+\O(T_{D^3(vw)} D^{-3}u)\\
& +T_u L(vw)+T_{Du}D^{-1}(vw)-\frac12T_{D^2u}D^{-2}(vw)+\O(T_{D^3u}D^{-3}(vw))\\
&+LR(u, vw),
\end{aligned}
\eeq
with
\[
\|LR(u, vw)\|_{\dot H^{s+2}}\leq C(\|u\|_{W^{2,\infty}} + \|Lu\|_{W^{2,\infty}})\|vw\|_{\dot H^s},
\]
where $\|Lu\|_{W^{2,\infty}} \le C\|u\|_{W^{3,\infty}}$ and
$\|vw\|_{\dot H^s} \le C\left(\|v\|_{L^\infty}\|w\|_{\dot H^s}+ \|w\|_{L^\infty}\|v\|_{\dot H^s}\right)$.

Using the fact that
\[
\|T_{vw}-T_vT_w\|_{\dot{H}^s \to \dot{H}^{s+\sigma}}\leq C(\|v\|_{W^{\sigma,\infty}}\|w\|_{L^\infty}
+\|v\|_{L^\infty}\|w\|_{W^{\sigma,\infty}}),
\]
and denoting the remainder terms by $\Rc_i$, we can expand each term in the above equation to get
\begin{align}
\begin{split}
T_{vw} Lu &=T_vT_w Lu+\Rc_1,\\
T_{D(vw)}D^{-1}u &=(T_{Dv}T_w +T_{v}T_{Dw})D^{-1}u+\Rc_2,\\
T_{D^2(vw)} D^{-2}u &=[T_{D^2v} T_w+T_{D^2w} T_v+2T_{Dv} T_{Dw}]D^{-2}u+\Rc_3,\\
D^{-1}(vw)&=D^{-1}(T_v w+T_wv+R(v,w))\\
&= T_vD^{-1}w-T_{Dv}D^{-2}w+\O(T_{D^2v}D^{-3}w)\\
&+T_wD^{-1}v-T_{Dw}D^{-2}v+\O(T_{D^2w}D^{-3}v)+\Rc_4,\\
D^{-2}(vw)&=D^{-2}(T_v w+T_wv+R(v,w))\\
&= T_vD^{-2}w-2T_{Dv}D^{-3}w+\O(T_{D^2v}D^{-4}w)\\
&+T_wD^{-2}v-2T_{Dw}D^{-3}v+\O(T_{D^2w}D^{-4}v)+\Rc_5,
\end{split}
\label{lem-luvw2}
\end{align}
with
\[
\begin{aligned}
\|\Rc_i\|_{\dot{H}^{s+2}} & \leq C \|v\|_{W^{3,\infty}}\|w\|_{W^{3,\infty}}\|u\|_{\dot{H}^s}, \quad \text{for}\ i = 1, 2, 3,\\
\|\Rc_4\|_{\dot{H}^{s+2}} & \leq C \|v\|_{W^{1,\infty}}\|w\|_{\dot{H}^s} + \|w\|_{W^{1,\infty}}\|v\|_{\dot{H}^s},\\
\|\Rc_5\|_{\dot{H}^{s+2}} &\leq \|v\|_{L^{\infty}}\|w\|_{\dot{H}^s} + \|w\|_{L^{\infty}}\|v\|_{\dot{H}^s}.
\end{aligned}
\]
Then the lemma is proved by substituting \eqref{lem-luvw2} into \eqref{lem-luvw1}.
\end{proof}

\begin{lemma}\label{lem-DsT}
If $u, v \in \D'(\T)$ and $s\in \R$, then
\[
|D|^sT_u v=T_u |D|^sv+sT_{Du}|D|^{s-2}Dv+\frac{s(s-1)}2 T_{|D|^2u}|D|^{s-2}v+\O(T_{|D|^3 u} |D|^{s-3} v).
\]
\end{lemma}
\begin{proof}
By the definition of Weyl para-product
\[
\begin{aligned}
\F (|D|^sT_u v)(\xi)&= \frac1{2\pi} |\xi|^s \sum\limits_{\eta\in\Z_*}\chi\left(\frac{|\xi-\eta|}{|\xi+\eta|}\right)\hat u(\xi-\eta)\hat v(\eta)
\\
&= \frac1{2\pi} \sum\limits_{\eta\in\Z_*}|\xi-\eta+\eta|^s\chi\left(\frac{|\xi-\eta|}{|\xi+\eta|}\right)\hat u(\xi-\eta)\hat v(\eta).
\end{aligned}
\]
As in the proof of Lemma \ref{lem-luv}, we have on the support of $\chi\left({|\xi-\eta|}/{|\xi+\eta|}\right)$ that
\[
\left|\frac{\xi-\eta}{\eta}\right| \leq \frac{2\ve}{1-\ve},
\]
and, using the Taylor expansion
\[
\begin{aligned}
|\xi-\eta+\eta|^s&=|\eta|^s\left|1+\frac{\xi-\eta}{\eta}\right|^s\\
&= |\eta|^s\left(1+s\frac{\xi-\eta}{\eta}+\frac{s(s-1)}2\frac{(\xi-\eta)^2}{\eta^2}+O\bigg(\frac{|\xi-\eta|^3}{\eta^3}\bigg)\right)
\end{aligned}
\]
in the expression for $\F(|D|^s T_u v)$, we get
\[
\begin{aligned}
\F & (|D|^s T_u v)(\xi) \\
&= \frac1{2\pi} \sum\limits_{\eta\in\Z_*} |\eta|^s\left(1+s\frac{\xi-\eta}{\eta}+\frac{s(s-1)}2\frac{(\xi-\eta)^2}{\eta^2}+O\bigg(\frac{|\xi-\eta|^3}{\eta^3}\bigg)\right)
\chi\left(\frac{|\xi-\eta|}{|\xi+\eta|}\right)\hat u(\xi-\eta)\hat v(\eta)\\
&= \F \left[T_u |D|^sv+sT_{Du}|D|^{s-2}Dv+\frac{s(s-1)}2 T_{|D|^2u}|D|^{s-2}v+\O(T_{|D|^3 u} |D|^{s-3} v)\right](\xi),
\end{aligned}
\]
which proves the lemma.
\end{proof}

\section{Bony decomposition of the equation}\label{sec-simeqn}

In this section, we carry out a Bony decomposition of the approximate SQG front equation
\begin{align}
\label{sqgeq}
\begin{split}
&\varphi_t + \frac 12 \partial_x \bigg\{\varphi^2 L \varphi_{xx}
- \varphi L (\varphi^2)_{xx} + \frac 13 L (\varphi^3)_{xx}\bigg\}
= 2 L \varphi_x,
\end{split}
\end{align}
where $L = \log|\partial_x|$, to put it in a form that allows us to make
weighted energy estimates. This form makes explicit the cancelation of second-order derivatives in the flux
and extracts a nonlinear term $L(T_{\vp_x}^2\vp)$ from the flux that is responsible for the logarithmic loss of derivatives in the dispersionless equation.

In the following, we use $\P(\cdot)$ to denote a nondecreasing polynomial, which might change from line to line.

\begin{lemma}\label{bony-sqg}
Suppose that $\vp(\cdot,t) \in \dot{H}^s(\T)$ with $s > 7/2$. Then \eqref{sqgeq} can be written as
\bel{SEq}
\vp_t+\px \left\{ \frac12T_{B(\vp)} \vp+[T_{\vp_x}, T_\vp]\vp_x \right\}+\Rc_7=L[(2-T_{\vp_x}^2)\vp]_x
\eeq
where
\begin{align}
B(\vp) &= \vp_x^2-3\vp\vp_{xx}-2\vp_{xx}L\vp-4\vp_xL\vp_x,
\label{defB}
\end{align}
and the remainder term $\Rc_7$ satisfies the estimate
\begin{equation}
\|\Rc_7\|_{\dot{H}^s}  \leq \P\left(\|\vp\|_{\dot{H}^s}\right)
\label{r7est}
\end{equation}
for a nondecreasing polynomial $\P$.
\end{lemma}

\begin{proof}
The nonlinear flux term in \eqref{sqgeq} is given by
\[
\varphi^2 L \varphi_{xx}
- \varphi L (\varphi^2)_{xx} + \frac 13 L (\varphi^3)_{xx}
= \vp^2 L \vp_{xx}- 2\vp L (\vp\vp_{xx}+\vp_x^2)+ L (\vp^2 \vp_{xx}+2\vp \vp_x^2).
\]
We will use the lemmas from last section to expand this term.

\noindent{\bf 1.  Term $L(\vp\vp_{xx}+\vp_x^2)$.}

By Lemma \ref{lem-luv}, we have that
\[
\begin{aligned}
L(\vp\vp_{xx})&=T_\vp L\vp_{xx}+T_{D\vp}D^{-1}\vp_{xx}-\frac12 T_{D^2\vp}D^{-2}\vp_{xx}+T_{\vp_{xx}}L\vp+\Rc_1,\\
L(\vp_x^2)&=2T_{\vp_x} L\vp_x+2 T_{D\vp_x}D^{-1}\vp_x+\Rc_2,
\end{aligned}
\]
with
\[
\begin{aligned}
\|\Rc_1\|_{\dot{H}^{s+1}} &\leq C \|\vp\|_{W^{3,\infty}}\|\vp\|_{\dot{H}^s},\qquad
\|\Rc_2\|_{\dot{H}^{s+1}} &\leq C \|\vp\|_{W^{3,\infty}}\|\vp\|_{\dot{H}^s}.
\end{aligned}
\]

\noindent{\bf 2. Term $L(\vp^2\vp_{xx})$.}

By Lemma \ref{lem-luvw}, we have that
\[\begin{aligned}
L(\vp^2\vp_{xx}) & = T_\vp T_\vp L\vp_{xx}+ 2T_{\vp}T_{\vp_{xx}} L\vp+2T_{D\vp}T_\vp D^{-1}\vp_{xx}-\frac{1}2 [2T_{D^2\vp}T_\vp+2T_{D\vp}T_{D\vp}]D^{-2}\vp_{xx}+\Rc_3,\\
& = T_\vp T_\vp L\vp_{xx}+ 2T_{\vp}T_{\vp_{xx}} L\vp+2T_{D\vp}T_\vp D^{-1}\vp_{xx}-[T_{D^2\vp}T_\vp+T_{D\vp}T_{D\vp}]D^{-2}\vp_{xx}+\Rc_3,
\end{aligned}
\]
with
\[
\|\Rc_3\|_{\dot{H}^{s+1}} \leq C \|\vp\|_{W^{3,\infty}}^2\|\vp\|_{\dot{H}^s}.
\]

\noindent{\bf 3. Term $L(\vp\vp_x^2)$.}

By Lemma \ref{lem-luvw}, we have that
\[
L(\vp\vp_x^2)= 2T_\vp T_{\vp_x} L\vp_x+T_{\vp_x} T_{\vp_x} L\vp+ 2(T_{D\vp}T_{\vp_x}+T_\vp T_{D\vp_x})D^{-1}\vp_x+\Rc_4,
\]
with
\[
\|\Rc_4\|_{\dot{H}^{s+1}} \leq C \|\vp\|_{W^{3,\infty}}^2\|\vp\|_{\dot{H}^s}.
\]

\noindent{\bf 4. Term $\vp^2L\vp_{xx}$.}

By Bony's decomposition, we can express $\vp^2L\vp_{xx}$ as
\[
\vp^2L\vp_{xx}= T_\vp T_\vp L\vp_{xx}+ 2 T_\vp T_{L\vp_{xx}} \vp+\Rc_5,
\]
with
\[
\|\Rc_5\|_{\dot{H}^{s+1}} \leq C \|L\vp\|_{W^{3,\infty}}^2\|\vp\|_{\dot{H}^s}.
\]
Collecting all the above expressions, we obtain that
\[
\begin{aligned}
\vp^2 L & \vp_{xx}- 2\vp L (\vp\vp_{xx}+\vp_x^2)+ L (\vp^2 \vp_{xx}+2\vp \vp_x^2)\\
&= T_\vp T_\vp L\vp_{xx}+ 2 T_\vp T_{L\vp_{xx}} \vp-2\vp \bigg[T_\vp L\vp_{xx}+T_{D\vp}D^{-1}\vp_{xx}-\frac12 T_{D^2\vp}D^{-2}\vp_{xx}\\
&\qquad+T_{\vp_{xx}}L\vp + 2T_{\vp_x} L\vp_x+2 T_{D\vp_x}D^{-1}\vp_x\bigg] +T_\vp T_\vp L\vp_{xx}+ 2T_{\vp}T_{\vp_{xx}} L\vp\\
&\qquad +2T_{D\vp}T_\vp D^{-1}\vp_{xx}-\big[T_{D^2\vp}T_\vp+T_{D\vp}T_{D\vp}\big]D^{-2}\vp_{xx}+4T_\vp T_{\vp_x} L\vp_x\\
&\qquad+2T_{\vp_x} T_{\vp_x} L\vp+ 4(T_{D\vp}T_{\vp_x}+T_\vp T_{D\vp_x})D^{-1}\vp_x+ \mathfrak{R}\\
&= T_\vp T_\vp L\vp_{xx}+ 2 T_\vp T_{L\vp_{xx}} \vp-2T_\vp \bigg[T_\vp L\vp_{xx}+T_{D\vp}D^{-1}\vp_{xx}-\frac12 T_{D^2\vp}D^{-2}\vp_{xx}\\
&\qquad+T_{\vp_{xx}}L\vp + 2T_{\vp_x} L\vp_x+2 T_{D\vp_x}D^{-1}\vp_x\bigg] -2T_{A}\vp
+T_\vp T_\vp L\vp_{xx}+ 2T_{\vp}T_{\vp_{xx}} L\vp\\
&\qquad +2T_{D\vp}T_\vp D^{-1}\vp_{xx}-[T_{D^2\vp}T_\vp+T_{D\vp}T_{D\vp}]D^{-2}\vp_{xx}+4T_\vp T_{\vp_x} L\vp_x\\
&\qquad+2T_{\vp_x} T_{\vp_x} L\vp+ 4(T_{D\vp}T_{\vp_x}+T_\vp T_{D\vp_x})D^{-1}\vp_x+ \mathfrak{R} -2R(A,\vp),
\end{aligned}
\]
where
\[
\begin{aligned}
\mathfrak{R} & = -2 \vp (\Rc_1 + \Rc_2) + \Rc_3 + 2 \Rc_4 + \Rc_5,\\
A & =T_\vp L\vp_{xx}+T_{D\vp}D^{-1}\vp_{xx}-\frac12 T_{D^2\vp}D^{-2}\vp_{xx}
+T_{\vp_{xx}}L\vp + 2T_{\vp_x} L\vp_x+2 T_{D\vp_x}D^{-1}\vp_x.
\end{aligned}
\]
Simplifying the above equation, we find that the higher order terms involving $L\vp_{xx}$ and $L\vp_x$ vanish, and
\[
\begin{aligned}
&\varphi^2 L \varphi_{xx}
- \varphi L (\varphi^2)_{xx} + \frac 13 L (\varphi^3)_{xx}\\
&\qquad=  2 T_\vp T_{L\vp_{xx}} \vp-2T_\vp \left[\frac12 T_{D^2\vp}\vp+T_{\vp_{xx}}L\vp +2 T_{\vp_{xx}}\vp\right] -2T_{A}\vp
+ 2T_{\vp}T_{\vp_{xx}} L\vp\\
&\qquad\qquad -\left[T_{\vp_{xx}}T_\vp+T_{\vp_x}T_{\vp_x}\right]\vp+2T_{\vp_x} T_{\vp_x} L\vp+ 4(T_{\vp_x}T_{\vp_x}+T_\vp T_{\vp_{xx}})\vp\\
&\qquad\qquad+ \mathfrak{R} -2R(A,\vp) + 2[T_{D\vp}, T_\vp] D^{-1}\vp_{xx}\\
&\qquad= 2T_{\vp_x} T_{\vp_x} L\vp+  2 T_\vp T_{L\vp_{xx}} \vp+T_\vp  T_{\vp_{xx}}\vp -2T_{A}\vp -T_{\vp_{xx}}T_\vp\vp+3T_{\vp_x}T_{\vp_x}\vp\\
&\qquad\qquad+\mathfrak{R} -2R(A,\vp) + 2[T_{D\vp}, T_\vp] D^{-1}\vp_{xx}\\
&\qquad= 2T_{\vp_x}^2 L\vp+ T_{B} \vp + 2[T_{D\vp}, T_\vp] D^{-1}\vp_{xx}+\tilde \Rc,
\end{aligned}
\]
where $B$ is given by \eqref{defB}, and
\begin{align*}
\tilde \Rc & = \mathfrak{R} -2R(A,\vp) +(\tilde B- T_{B})\vp,\\
 \tilde B &= 2 T_\vp T_{L\vp_{xx}} +T_\vp T_{\vp_{xx}} -2T_{A} -T_{\vp_{xx}}T_\vp+3T_{\vp_x}T_{\vp_x}.
\end{align*}
By a Kato-Ponce type commutator estimate (see e.g., \cite{kato-ponce}), we have
\[
\|[T_{D\vp}, T_\vp] D^{-1}\vp_{xx}\|_{\dot{H}^{s + 1}} \leq C \|\vp\|_{W^{2,\infty}}^2 \|\vp\|_{\dot{H}^{s + 1}}.
\]
In addition, using the estimates of the remainders $\Rc_i$,  Sobolev embedding, and the estimate
\[
\|(\tilde B- T_{B})\vp\|_{\dot{H}^{s + 1}} \leq C (\|\vp\|_{W^{3,\infty}}^2) \|\vp\|_{\dot{H}^{s}},
\]
we get that
\[
\|\tilde \Rc\|_{\dot{H}^{s + 1}} \leq C \|\vp\|_{\dot{H}^{s}}^3.
\]
It follows that \eqref{sqgeq} can be written as
\[
\vp_t+\frac12\px \left\{2T_{\vp_x}^2L\vp+ T_{B} \vp+2[T_{\vp_x}, T_\vp]\vp_x +\tilde \Rc\right\}=2L\vp_x.
\]
or
\begin{equation}
\vp_t+\px \left\{\frac12T_{B} \vp+[T_{\vp_x}, T_\vp]\vp_x \right\}+\Rc_6=[(2-T_{\vp_x}^2)L\vp]_x,
\label{vpeq}
\end{equation}
where $\|\Rc_6\|_{\dot{H}^s} \leq \P( \|\vp\|_{\dot{H}^s})$ for $s > 7 / 2$.

Using Lemma \ref{lem-luv} to expand the term $L (2 - T_{\vp_x}^2) \vp$, we have the commutator estimate
\[
\|[(2-T_{\vp_x}^2), L]\vp\|_{H^{s+1}} \leq \P(\|\vp\|_{\dot{H}^s}).
\]
Hence, we can rewrite \eqref{vpeq} as \eqref{SEq},
with $L[(2-T_{\vp_x}^2)\vp]_x$ as the highest order term,
\[
\px \left\{ \frac12T_{B} \vp+[T_{\vp_x}, T_\vp]\vp_x \right\}
\]
as the first order term, and $\Rc_7$ as the zeroth order term, which satisfies \eqref{r7est} and does not lose derivatives.
\end{proof}

\section{Energy Estimate}\label{sec-apriori}

In this section, we prove an \emph{a priori} estimate for the initial value problem \eqref{sqgivp},
which is stated in Proposition \ref{apriori} below.

We first recall the following definition for fractional powers of operators. If $T \colon \H \to \H$
is a self-adjoint linear operator on a Hilbert space $\H$
and $f \in C_c^\infty(\R)$ is a function,
then $f(T)$ may be defined by the Helffer-Sj\"ostrand formula \cite{Dav, Hel} as
\begin{align}\label{spec-def}
\begin{split}
f(T) &= - \frac{1}{\pi} \lim_{\epsilon \to 0^+} \int_{|\Im z| > \epsilon} \partial_{\bar z} \tilde f(z) (z - T)^{-1} \diff{\alpha} \diff{\beta},
\\
\tilde f(z) &= \left(f(\alpha) + i \beta f'(\alpha) + \frac 12 (i\beta)^2 f''(\alpha)\right) \chi_0(\beta),
\end{split}
\end{align}
where $z = \alpha + i\beta$, $\partial_{\bar z} = \frac 12 (\partial_\alpha + i \partial_\beta)$, and
the cutoff-function $\chi_0\in C_c^\infty(\R)$ is equal to $1$ in a neighborhood of $0$.
The function $\tilde f$ is an ``almost analytic'' extension of $f$ since
\begin{align}\label{almost-ana}
\begin{split}
\partial_{\bar z} \tilde f(z)&
= O(|\Im z|^2)
\qquad\text{as $\Im z \to 0$ with $\Re z$ fixed}.
\end{split}
\end{align}
Furthermore, if $U\subset \R$ is an open set that contains the spectrum $\sigma(T)\subset \R$ of $T$ and
$g \in C^\infty(U)$, then, by the resolution of identity form of the spectral theorem \cite{ReSi},
we see that $g(T) = f(T)$, where $f =  g \chi_1$ and $\chi_1 \in C_c^\infty(U)$ with
$\chi_1 = 1$ on $\sigma(T)$.

In particular, if $\|T_{\vp_x}^2\|_{L^2\to L^2} < 2$,
then $(2 - T_{\vp_x}^2)$ is a positive, self-adjoint operator on $L^2$, and
$(2-T_{\vp_x}^2)^s$ is well-defined for $s\in \R$ by \eqref{spec-def} as $f(2 - T_{\vp_x}^2)$,
where
\begin{equation}
f(\alpha) = |\alpha|^{s} \chi_1(\alpha)
\label{deffalpha}
\end{equation}
for $\chi_1\in C_c^\infty(0,2)$ such that $\chi_1 = 1$ on $\sigma(2 - T_{\vp_x}^2)$.
We can therefore define a weighted $s$-order energy  by
\begin{equation}
E^{(s)}(t)=\int_{\T} |D|^{s} \vp(x,t)\cdot\left(2-T_{\vp_x(x,t)}^2\right)^{2s+1} |D|^{s} \vp(x,t) \diff{x}.
\label{weighted_energy1}
\end{equation}

In order to prove Proposition \ref{apriori}, we need the following lemma.
\begin{lemma}\label{lem-dfT} Suppose that $s > 7 / 2$.
If $\vp$ is a smooth solution of \eqref{SEq}
and $\psi \in L^2$, then
\[
\begin{aligned}
\partial_t (2 - T_{\vp_x}^2)^s \psi & = (2 - T_{\vp_x}^2)^s \psi_t - s (2 - T_{\vp_x}^2)^{s - 1} (T_{\vp_x} T_{\vp_{xt}} + T_{\vp_{xt}} T_{\vp_x}) \psi + \Rc_8(\psi),
\end{aligned}
\]
where the remainder term satisfies
\[
\|\Rc_8(\psi)\|_{\dot H^1} \leq \P\left(\|\vp\|_{\dot H^s}\right) \|\psi\|_{L^2}
\]
for a nondecreasing polynomial $\P$.
\end{lemma}
\begin{proof}
For $z = \alpha + i \beta  \in \C \setminus \R$, we have
\begin{align*}
&\left[\pt (z - 2 + T_{\vp_x}^2)^{-1}\right] (z - 2 + T_{\vp_x}^2) + (z - 2 + T_{\vp_x}^2)^{-1} (T_{\vp_x} T_{\vp_{xt}} + T_{\vp_{xt}} + T_{\vp_x})
\\
 &\qquad\qquad = \pt \left[(z - 2 + T_{\vp_x}^2)^{-1} (z - 2 + T_{\vp_x}^2)\right] = \pt\,{\rm Id}
 = 0.
\end{align*}
It follows that
\[
\begin{aligned}
\pt (z - 2 + T_{\vp_x}^2)^{-1} & = - (z - 2 + T_{\vp_x}^2)^{-1} (T_{\vp_x} T_{\vp_{xt}} + T_{\vp_{xt}} T_{\vp_x}) (z - 2 + T_{\vp_x}^2)^{-1}\\
& = - (z - 2 + T_{\vp_x}^2)^{-2} (T_{\vp_x} T_{\vp_{xt}} + T_{\vp_{xt}} + T_{\vp_x})\\
& \qquad + (z - 2 + T_{\vp_x}^2)^{-1} \left[(z - 2 + T_{\vp_x}^2)^{-1}, T_{\vp_x} T_{\vp_{xt}} + T_{\vp_{xt}} + T_{\vp_x}\right]\\
& = \partial_z (z - 2 + T_{\vp_x}^2)^{-1} (T_{\vp_x} T_{\vp_{xt}} + T_{\vp_{xt}} + T_{\vp_x})\\
& \qquad + (z - 2 + T_{\vp_x}^2)^{-1} \left[(z - 2 + T_{\vp_x}^2)^{-1}, T_{\vp_x} T_{\vp_{xt}} + T_{\vp_{xt}} + T_{\vp_x}\right].
\end{aligned}
\]
Using \eqref{spec-def}, where $f$ is defined by \eqref{deffalpha}, and the previous equation, we get that
\[
\begin{aligned}
\pt (2 - T_{\vp_x}^2)^s \psi & = \pt f(2 - T_{\vp_x}^2)\psi\\
& = (2 - T_{\vp_x}^2)^s \psi_t - \frac{1}{\pi} \left[\lim_{\epsilon \to 0^+} \int_{|\Im z| > \epsilon} \partial_{\bar z} \tilde f(z) \pt (z - 2 + T_{\vp_x}^2)^{-1} \diff{\alpha} \diff{\beta}\right] \psi\\
& = (2 - T_{\vp_x}^2)^s \psi_t + T_1 \psi + \Rc_8,
\end{aligned}
\]
where
\begin{align*}
T_1 \psi &= - \frac{1}{\pi} \left[\lim_{\epsilon \to 0^+} \int_{|\Im z| > \epsilon} \partial_{\bar z} \tilde f(z) \partial_z (z - 2 + T_{\vp_x}^2)^{-1} \diff{\alpha} \diff{\beta}\right] (T_{\vp_x} T_{\vp_{xt}} + T_{\vp_{xt}} T_{\vp_x}) \psi,
\\
\Rc_8 &=  - \frac{1}{\pi} \bigg[\lim_{\epsilon \to 0^+} \int_{|\Im z| > \epsilon} \partial_{\bar z} \tilde f(z) (z - 2 + T_{\vp_x}^2)^{-1} \big[(z - 2 + T_{\vp_x}^2)^{-1}, T_{\vp_x} T_{\vp_{xt}} + T_{\vp_{xt}} + T_{\vp_x}\big] \diff{\alpha} \diff{\beta}\bigg] \psi.
\end{align*}
Since $2 - T_{\vp_x}^2$ is self-adjoint, we have $\partial_{\bar z} (z - 2 + T_{\vp_x}^2)^{-1} = 0$ for
$z \in \C \setminus \R$, so
\[
\partial_z (z - 2 + T_{\vp_x}^2)^{-1} =  \partial_\alpha (z - 2 + T_{\vp_x}^2)^{-1}.
\]
We can then integrate by parts with respect to $\alpha$ in $T_1 \psi$ to get
\[
\begin{aligned}
T_1 \psi
& = \frac{1}{\pi} \left[\lim_{\epsilon \to 0^+} \int_{|\Im z| > \epsilon} \partial_{\bar z} \widetilde {f'}(z) (z - 2 + T_{\vp_x}^2)^{-1} \diff{\alpha} \diff{\beta}\right] (T_{\vp_x} T_{\vp_{xt}} + T_{\vp_{xt}} T_{\vp_x}) \psi\\
& = - s (2 - T_{\vp_x}^2)^{s- 1} (T_{\vp_x} T_{\vp_{xt}} + T_{\vp_{xt}} T_{\vp_x}) \psi.
\end{aligned}
\]

Finally, using a Kato-Ponce type estimate for commutators and \eqref{SEq} to estimate $\vp_{xt}$, we have
\[
\bigg\|(z - 2 + T_{\vp_x}^2)^{-1} \big[(z - 2 + T_{\vp_x}^2)^{-1}, T_{\vp_x} T_{\vp_{xt}} + T_{\vp_{xt}} T_{\vp_x}\big]\bigg\|_{L^2 \to \dot H^1} \leq \P\left(\|\vp\|_{\dot H^s}\right) |\Im z|^{-2}.
\]
It follows that
\[
\|\Rc_8\|_{\dot H^1} \leq \P\left(\|\vp\|_{\dot H^s}\right)  \|\psi\|_{L^2} \left[\lim_{\epsilon \to 0^+} \int_{|\Im z| > \epsilon} |\partial_{\bar z} \tilde f(z)| |\Im z|^{-2} \diff{\alpha} \diff{\beta} \right],
\]
where the integral converges by \eqref{almost-ana}.
\end{proof}

We now prove the following \emph{a priori} estimate.

\begin{proposition}\label{apriori}
Suppose that $s>7/2$ and $\vp$ is a smooth solution of \eqref{sqgivp} with $\vp_0 \in \dot{H}^s$. If $\|T_{\vp_{0x}}^2\|_{L^2 \to L^2} \leq C$ for some constant $0 < C < 2$, then
there exists a time $\Time > 0$ and  a constant $M > 0$, depending on $\vp_0$, such that
\[
\sup_{t \in [0, \Time]}E^{(s)}(t) \leq M,
\]
where $E^{(s)}(t)$ is defined in \eqref{weighted_energy1}.
\end{proposition}

\begin{proof}
We apply the operator $|D|^s$ to equation \eqref{SEq} to get
\beq\label{Dseqn}
|D|^s \vp_t +  |D|^s \px \left( \frac12T_{B} \vp+[T_{\vp_x}, T_\vp]\vp_x \right) + |D|^s\Rc_7 = \px L|D|^s \left[(2-T_{\vp_x}^2) \vp\right].
\eeq
Using Lemma \ref{lem-DsT} twice, we find that
\[
\begin{aligned}
|D|^s\left[(2-T_{\vp_x}^2) \vp\right]&=2|D|^s\vp-|D|^s(T_{\vp_x}^2\vp)\\
&=2|D|^s\vp-T_{\vp_x}^2 |D|^s\vp+sT_{\vp_x}T_{\vp_{xx}}|D|^{s-2}\vp_x+sT_{\vp_{xx}}T_{\vp_{x}}|D|^{s-2}\vp_x+\Rc_{10},
\end{aligned}
\]
where $\|\px \Rc_{10}\|_{L^2} \leq C \|\vp\|_{W^{3,\infty}}^2\|\vp\|_{\dot{H}^s}$.

Thus, we write can the right-hand side of \eqref{Dseqn} as
\[
\begin{aligned}
\px L |D|^s & \left[(2 - T_{\vp_x}^2) \vp\right]\\
& = \px L \left[(2-T_{\vp_x}^2) |D|^s\vp+sT_{\vp_x}T_{\vp_{xx}}|D|^{s-2}\vp_x+sT_{\vp_{xx}}T_{\vp_{x}}|D|^{s-2}\vp_x\right]+\Rc_{11}\\
& =L\big\{(2-T_{\vp_x}^2) |D|^s\vp_x-T_{\vp_x}T_{\vp_{xx}}|D|^s\vp-T_{\vp_{xx}}T_{\vp_{x}}|D|^s\vp\\
& \hspace{1.5in} -sT_{\vp_x}T_{\vp_{xx}}|D|^{s}\vp-sT_{\vp_{xx}}T_{\vp_{x}}|D|^{s}\vp\big\}+\Rc_{12}\\
&= L\left\{(2-T_{\vp_x}^2) |D|^s\vp_x-(s+1)(T_{\vp_x}T_{\vp_{xx}}+T_{\vp_{xx}}T_{\vp_{x}})|D|^s\vp\right\}+\Rc_{12}.
\end{aligned}\]
Applying $(2-T_{\vp_x}^2)^s$ to \eqref{Dseqn}, and
commuting $(2-T_{\vp_x}^2)^s$ with $L$ up to a remainder
term, as in the proof of  Lemma \ref{lem-luv}, we obtain that
\begin{align}
\begin{split}
&(2-T_{\vp_x}^2)^s|D|^s \vp_t  +  (2-T_{\vp_x}^2)^s |D|^s \px \left( \frac12T_{B} \vp+[T_{\vp_x}, T_\vp]\vp_x \right)  \\
&\qquad=  L\left\{(2-T_{\vp_x}^2)^{s+1}|D|^s\vp_x-(s+1)(2-T_{\vp_x})^s(T_{\vp_x}T_{\vp_{xx}}
+T_{\vp_{xx}}T_{\vp_{x}})|D|^s\vp\right\}+\Rc_{13}
\\
&\qquad=  \px L\left\{(2-T_{\vp_x}^2)^{s+1}|D|^s\vp\right\}+\Rc_{14},
\end{split}
\label{JDseqn}
\end{align}
where $\|\Rc_{14}\|_{L^2} \leq \P(\|\vp\|_{\dot{H}^s})$.

By Lemma \ref{lem-dfT}, with $\psi=|D|^s\vp$, the time derivative of $E^{(s)}(t)$ in \eqref{weighted_energy1} is
\bel{dtE}\begin{aligned}
\frac{\diff}{\diff{t}} E^{(s)}(t) & = - \int_{\T} (2s+1) |D|^s\vp\cdot (2-T_{\vp_x}^2)^{2s} (T_{\vp_x}T_{\vp_{xt}} + T_{\vp_{xt}}T_{\vp_x}) |D|^s\vp\diff{x}\\
& + 2\int_{\T}|D|^s\vp \cdot (2-T_{\vp_x}^2)^{2s+1} |D|^s\vp_t\diff{x} + \int_{\T}\Rc_8\left(|D|^s\vp\right) |D|^s\vp \diff{x}.
\end{aligned}\eeq

\noindent {\bf(1)}  Equation \eqref{SEq} implies that $\|\vp_{xt}\|_{L^\infty} \leq  \P(\|\vp\|_{\dot{H}^s})$, so the first term on the right-hand side of \eqref{dtE} can be estimated by
\[
\left| \int_{\T} (2s+1)|D|^s\vp \cdot (2-T_{\vp_x}^2)^{2s}(T_{\vp_x}T_{\vp_{xt}} + T_{\vp_{xt}}T_{\vp_x}) |D|^s\vp\diff{x} \right| \leq C \|\vp\|_{W^{1,\infty}}^3\|\vp_{xt}\|_{L^\infty} \|\vp\|_{\dot{H}^s}^2 \leq \P(\|\vp\|_{\dot{H}^s}).
\]
In addition, from Lemma~\ref{lem-dfT}, the third term  on the right-hand side of \eqref{dtE} can be estimated by
\[
 \int_{\T}\Rc_8\left(|D|^s\vp\right) |D|^s\vp \diff{x} \le \P\left(\|\vp\|_{\dot{H}^s}\right).
\]
\noindent {\bf(2)} To estimate the second term  on the right-hand side \eqref{dtE}, we multiply \eqref{JDseqn} by $(2-T_{\vp_x}^2)^{s+1}|D|^s\vp$, integrate the result with respect to $x$, and use the self-adjointness of
$(2-T_{\vp_x}^2)^{s+1}$, which gives
\[
\int_{\T} |D|^s\vp \cdot (2-T_{\vp_x}^2)^{2s+1} |D|^s\vp_t \diff{x}
= \Rm{1} + \Rm{2} + \Rm{3},
\]
where
\begin{align*}
\Rm{1} &= - \int_{\T} |D|^s\vp\cdot (2-T_{\vp_x}^2)^{2s+1} |D|^s\partial_x\left( \frac12T_{B} \vp+[T_{\vp_x}, T_\vp]\vp_x \right) \diff{x},
\\
\Rm{2} &= \int_\T (2-T_{\vp_x}^2)^{s+1}|D|^s\vp\cdot \px L (2-T_{\vp_x}^2)^{s+1}|D|^s\vp \diff{x},
\\
\Rm{3}&=\int _\T(2-T_{\vp_x}^2)^{s+1}|D|^s\vp\cdot \Rc_{14} \diff{x}.
\end{align*}
We have $\Rm{2}=0$, since $\partial_x L$ is skew-symmetric, and
\[
\Rm{3} \leq \P(\|\vp\|_{\dot{H}^s}),
\]
since $\|\Rc_{14}\|_{L^2} \leq \P(\|\vp\|_{\dot{H}^s})$ and $(2 - T_{\vp_x}^2)^{s+1}$ is bounded on $L^2$.

\noindent {\bf Term} $\Rm{1}$ {\bf estimate.} We write  $\Rm{1} = -\Rm{1}_a + \Rm{1}_b$, where
\begin{align*}
\Rm{1}_a & =\int_{\T}|D|^s\vp\cdot (2-T_{\vp_x}^2)^{2s+1} \partial_x \left( \frac12T_{B} |D|^s\vp+[T_{\vp_x}, T_\vp]|D|^s\vp_x \right)\diff{x},\\
 \Rm{1}_b&= \int_{\T} |D|^s\vp\cdot (2-T_{\vp_x}^2)^{2s+1} \partial_x \left( \frac12[T_{B}, |D|^s]\vp+\left[[T_{\vp_x}, T_\vp],|D|^s\right]\vp_x \right) \diff{x}.
\end{align*}
By a commutator estimate, the second integral satisfies
$|\Rm{1}_b| \leq \P(\|\vp\|_{\dot{H}^s})$.

To estimate the first integral, we write it as
\begin{align*}
\Rm{1}_a &= \Rm{1}_{a_1} - \frac12 \Rm{1}_{a_2} - \Rm{1}_{a_3},
\end{align*}
where
\begin{align*}
\Rm{1}_{a_1}& =\int_{\T} |D|^s\vp\cdot [(2-T_{\vp_x}^2)^{2s+1},\partial_x]  \left( \frac12T_{B} |D|^s\vp+[T_{\vp_x}, T_\vp]|D|^s\vp_x \right) \diff{x},
\\
\Rm{1}_{a_2}  &= \int_{\T} |D|^s\vp_x \cdot(2-T_{\vp_x}^2)^{2s+1} \left(T_{B} |D|^s\vp\right) \diff{x},
\\
\Rm{1}_{a_3} &= \int_{\T} |D|^s\vp_x\cdot (2-T_{\vp_x}^2)^{2s+1}  \left([T_{\vp_x}, T_\vp]|D|^s\vp_x\right)\diff{x}.
\end{align*}

\noindent {\bf Term} $\Rm{1}_{a_1}$ {\bf estimate.} A Kato-Ponce commutator estimate gives
\[
|\Rm{1}_{a_1}| \leq \P\left(\|\vp\|_{\dot{H}^s}\right).
\]

\noindent {\bf Term} $\Rm{1}_{a_2}$ {\bf estimate.} We have
\bel{eqBs}\begin{aligned}
\Rm{1}_{a_2}
&=\int_{\T} \left(T_{B} |D|^s\vp\right) \cdot (2-T_{\vp_x}^2)^{2s+1}|D|^s\vp_x \diff{x}\\
&=\int_{\T} \left(T_{B} |D|^s\vp\right) \cdot \left\{ \partial_x\left( (2-T_{\vp_x}^2)^{2s+1}|D|^s\vp\right)-\left[\partial_x,(2-T_{\vp_x}^2)^{2s+1}\right] |D|^s\vp \right\}\diff{x}\\
&=- \int_{\T} \partial_x \left(T_{B} |D|^s\vp\right) \cdot (2-T_{\vp_x}^2)^{2s+1}|D|^s\vp \diff{x} - \int_\T \left(T_{B} |D|^s\vp\right) \cdot \left[\partial_x,(2-T_{\vp_x}^2)^{2s+1}\right] |D|^s\vp \diff{x}\\
&=- \int_{\T} \left(T_B |D|^s\vp_x + \left[\partial_x,T_{B}\right] |D|^s\vp\right) \cdot(2-T_{\vp_x}^2)^{2s+1}|D|^s\vp \diff{x}\\
& \qquad - \int_\T T_{B} \left(|D|^s\vp\right) \cdot \left[\partial_x,(2-T_{\vp_x}^2)^{2s+1}\right] |D|^s\vp \diff{x}.
\end{aligned}\eeq
Using the commutator estimates
\[
\left\|\left[\partial_x,(2-T_{\vp_x}^2)^{2s+1}\right] |D|^s\vp\right\|_{L^2} \leq \P(\|\vp\|_{\dot{H}^s}),
\qquad
\left\|\left[\partial_x,T_{B}\right] |D|^s\vp\right\|_{L^2} \leq \P(\|\vp\|_{\dot{H}^s}),
\]
and the fact that $T_B$ is self-adjoint, we can rewrite \eqref{eqBs} as
\[
\begin{aligned}
\Rm{1}_{a_2}
& = - \Rm{1}_{a_2} - \int_\T  |D|^s\vp\cdot \partial_x\left[(2-T_{\vp_x}^2)^{2s+1},T_B\right]|D|^s\vp \diff{x}+\Rc_{15},
\end{aligned}\]
with $|\Rc_{15}| \leq \P(\|\vp\|_{H^s})$.
Using the commutator estimate
\[
\left\|\partial_x \left[(2-T_{\vp_x}^2)^{2s+1},T_B\right]|D|^s\vp \diff{x}\right\|_{L^2} \leq \P(\|\vp\|_{\dot{H}^s}),
\]
we conclude that $|\Rm{1}_{a_2}| \leq \P(\|\vp\|_{\dot{H}^s})$.

\noindent {\bf Term} $\Rm{1}_{a_3}$ {\bf estimate.} Using the self-adjointness of $T_{\vp_x}$ and $T_{\vp}$, we obtain that
\[\begin{aligned}
\Rm{1}_{a_3}
& =\int_{\T}(2-T_{\vp_x}^2)^{2s+1}|D|^s\vp_x\cdot [T_{\vp_x}, T_{\vp}]|D|^s\vp_x \diff{x}\\
& =-\int_{\T} [T_{\vp_x}, T_{\vp}](2-T_{\vp_x}^2)^{2s+1}|D|^s\vp_x\cdot|D|^s\vp_x
\diff{x}.
\end{aligned}\]
Since
\[
\left\|\left[[T_{\vp_x}, T_{\vp}],(2-T_{\vp_x}^2)^{2s+1}\right]|D|^s\vp_x \right\|_{L^2} \leq \P(\|\vp\|_{\dot{H}^s}),
\]
we have that $|\Rm{1}_{a_3}|  \leq \P(\|\vp\|_{\dot{H}^s})$.

Collecting the above estimates, we obtain that
\[
\frac{\diff}{\diff{t}} E^{(s)} \leq \P\left(\|\vp\|_{\dot{H}^s}\right).
\]
Finally, since $\|2 - T_{\vp_{0x}}^2\|_{L^2 \to L^2} \ge 2-C$ and  $\|\vp_{x}(t)\|_{L^\infty}$ is continuous in time, there exists $\Time>0$ and $m>0$, depending only on the initial data, such that
\[
\|2-T_{\vp_x}^2\|_{L^2 \to L^2} \geq m \qquad \text{for $t\le\Time$}.
\]
We therefore obtain that
\[
m^{2s+1}\||D|^s\vp\|_{L^2}^2\leq E^{(s)}\leq 2^{2s+1} \||D|^s\vp\|_{L^2}^2,
\]
which implies that
\[
\frac{\diff}{\diff{t}}E^{(s)} \leq \P(E^{(s)}).
\]
The result then follows by Gr\"onwall's inequality.
\end{proof}

\section{Well-posedness}\label{Wp}

In this section, we construct solutions of \eqref{sqgivp} by a Galerkin method. For $N\in \N$,
let
\begin{equation}
J_N : L^2(\T) \to L^2(\T),\qquad J_N f(x)=\sum\limits_{|\xi| \le N} \hat f(\xi) e^{i\xi x}
\label{defJN}
\end{equation}
denote the projection onto the first $N$ Fourier modes.
We define an approximate solution $\vp^N(x,t)$ as the solution of the ODEs obtained by projection of \eqref{bony-sqg},
\beq
\vp^N_t+\px J_N\left\{\frac12T_{B(\vp^N)} \vp^N+[T_{\vp^N_x}, T_{\vp^N}]\vp^N_x \right\}+J_N\Rc_7(\vp^N)=J_NL[(2-T_{\vp^N_x}^2)\vp^N]_x,
\label{Eqapp}
\eeq
with initial data $\vp^N(x,0)=J_N\vp_0(x)$.

Repeating the previous estimates, we obtain that
\[
\frac{\diff}{\diff{t}} E^{(s)}(\vp^N) \leq \P\left(E^{(s)}(\vp^N)\right).
\]
Thus, since $E^{(s)}(J_N\vp_0) \lesssim \|\vp_0\|_{\dot{H}^s}^2$, there exists $T>0$ independent of $N$ 
such that the solution of \eqref{Eqapp} exists for $t\in[0,T]$ and
\[
\|\vp^N(t)\|_{\dot{H}^s}\leq 
\P(\|\vp_0\|_{\dot{H}^s}),
\]
where $\P$ is an nondecreasing polynomial independent of $N$.
The sequence of approximate solutions $\{\vp^N\}$ is therefore bounded in $L^\infty(0,T; \dot{H}^s)$,
so a subsequence converges weak-$*$ to a limit
\[
\vp \in L^\infty(0,T; \dot{H}^s).
\]
Moreover, from \eqref{Eqapp}, we see that $\{\vp_t^N\}$ is bounded in $L^\infty(0, T; \dot{H}^{s-1-\delta})$ for $\delta > 0$. The Aubin-Lions Lemma (see e.g., \cite{Ama00})
implies that a further subsequence converges strongly to $\vp$ in $C([0,T];\dot{H}^{r})$ for any $r < s$.
Taking the limit of \eqref{Eqapp} as $N\to\infty$, we find that $\vp$ is a solution of \eqref{bony-sqg}.

Since $\vp \in   L^\infty(0,T; \dot{H}^s) \cap C([0,T];\dot{H}^{r})$, we see that $\vp\in C_w([0,T];\dot{H}^{s})$ is weakly continuous in $\dot{H}^s$. In addition, the Arzel\`a-Ascoli theorem implies that $E^{(s)}(\vp)$ is continuous in time, since
$E^{(s)}(\vp^N)$ is continuous for each $N\in\N$, and
\[
\frac{\diff}{\diff{t}} E^{(s)}(\vp^N)
\]
is bounded uniformly in $N$. It follows that $\|\vp\|_{\dot{H}^s}$ is continuous, so,
by weak continuity and norm continuity, $\vp\in C([0,T]; \dot{H}^s)$
is strongly continuous in $\dot{H}^s$.


To prove the Lipschitz continuity \eqref{stability} and uniqueness, we suppose that $\vp, \psi\in C([0,T];\dot{H}^{s})$ are solutions of \eqref{sqgivp}
with $s>7/2$.
Subtracting the evolution equations for $\vp$ and $\psi$,
we find that $\Phi=\vp-\psi$ satisfies
\begin{align}
\begin{split}
\partial_t\Phi+\partial_x\left\{\frac12 T_{B(\vp)}\Phi+[T_{\vp_x},T_\vp]\Phi_x\right\}+\partial_x\left\{\frac12 [T_{B(\vp)}-T_{B(\psi)}]\psi+\Big[[T_{\vp_x},T_\vp]-[T_{\psi_x},T_\psi]\Big]\psi_x\right\}\\
=L[(2-T_{\vp_x}^2)\Phi]_x-(LT_{\vp_x}^2-LT_{\psi_x}^2)\psi_x+\Rc_7(\vp)-\Rc_7(\psi).
\end{split}
\label{Phieq}
\end{align}
For $r\ge 0$, we define a weighted $\dot{H}^r$-norm by
\begin{equation*}
E_\vp^{(r)}(\Phi(t))=\int_{\T} |D|^{r} \Phi(x,t)\cdot\left(2-T_{\vp_x(x,t)}^2\right)^{2r+1} |D|^{r} \Phi(x,t) \diff{x}.
\end{equation*}
Applying $\partial_x^{r}$ to \eqref{Phieq}, with $0\le  r < s-1$, and carrying out energy estimates as before, we get
\[
\frac \diff{\diff t} E_{\vp}^{(r)}(\Phi)\leq \P( \|\vp\|_{H^s}, \|\psi\|_{H^s}) \big[ E_{\vp}^{(r)}(\Phi)+\|L\partial_x^{r+1}\psi\|_{L^\infty}\|\Phi\|_{H^r}^2\big],
\]
where we have used the estimates
\begin{align*}
\|\partial_x^r(LT_{\vp_x}^2-LT_{\psi_x}^2)\psi_x\|_{L^2}&\lesssim\left\{\begin{array}{ll} \|\Phi_x\|_{L^\infty}(\|\vp_x\|_{L^\infty}+\|\psi_x\|_{L^\infty})\|L\partial_x^{r+1}\psi\|_{L^2},& \text{ when } \frac32< r <s-1,\\
\|\Phi_x\|_{L^2}(\|\vp_x\|_{L^\infty}+\|\psi_x\|_{L^\infty})\|L\partial_x^{r+1}\psi\|_{L^\infty}, & \text{ when } 1\leq r\leq \frac32,\\
\|\Phi\|_{\dot H^r} (\|\vp_x\|_{L^\infty}+\|\psi_x\|_{L^\infty})\|L\partial_x^{2}\psi\|_{L^\infty},         & \text{ when } 0\leq r<1,
\end{array}\right.\\
&\lesssim \|\Phi\|_{H^r}(\|\vp_x\|_{L^\infty}+\|\psi_x\|_{L^\infty})\|\psi\|_{H^s},
\\
\|\partial_x^r [\Rc_7(\vp)-\Rc_7(\psi)]\|_{L^2}&\lesssim \P(\|\vp\|_{H^s}, \|\psi\|_{H^s})\|\Phi\|_{\dot H^r}.
\end{align*}
It follows that
\[
E_{\vp}^{(0)}(\Phi(t))+E_{\vp}^{(r)}(\Phi(t))\lesssim \left[E_{\vp}^{(0)}(\Phi(0))+E_{\vp}^{(r)}(\Phi(0))\right] \int_0^t \P( \|\vp\|_{H^s}, \|\psi\|_{H^s}) \diff t,
\]
and, since $E_{\vp}^{(0)}(\Phi)+E_{\vp}^{(r)}(\Phi)$ is equivalent to $\|\Phi\|^2_{H^r}$,  the solution map is Lipschitz continuous on $\dot{H}^r$. In particular, the solution is unique.

Finally, we prove that the solution map is continuous on $\dot{H}^s$  by a Bona-Smith argument  \cite{bonasmith}.
First, suppose that $\vp\in C([0,T]; \dot{H}^s)$,
$\psi \in  C([0,T]; \dot{H}^{s+1+\delta})$
are solutions, where $0<\delta \ll 1$, and let $\Phi = \vp-\psi$. In a similar way to before, we find that $E_\vp^{(s)}(\Phi)$ satisfies
\begin{align*}
\frac{\diff}{\diff t} E^{(s)}_{\vp}(\Phi) &\leq \P( \|\vp\|_{\dot{H}^s}, \|\psi\|_{\dot{H}^s})E_{\vp}^{(s)}(\Phi)
\\
&\qquad +\|2-T_{\vp_x}^2\|_{L^\infty}^{2s+1}\Big[\|\partial_x^s(LT_{\vp_x}^2-LT_{\psi_x}^2)\psi_x\|_{L^2}+\|\partial_x^s [\Rc_7(\vp)-\Rc_7(\psi)]\|_{L^2}\Big]\|\Phi\|_{\dot H^s}.
\end{align*}
Using the estimates
\begin{align*}
\|\partial_x^s(LT_{\vp_x}^2-LT_{\psi_x}^2)\psi_x\|_{L^2}&\lesssim \|L\psi\|_{\dot{H}^{s+1}}\|\Phi\|_{\dot{H}^{2}}(\|\vp\|_{\dot{H}^s}+ \|\psi\|_{\dot{H}^s}),
\\
\|\partial_x^s [\Rc_7(\vp)-\Rc_7(\psi)]\|_{L^2}&\lesssim \P(\|\vp\|_{\dot{H}^s}, \|\psi\|_{\dot{H}^s})\|\Phi\|_{\dot H^s},
\end{align*}
we get that
\begin{equation}\label{ContD}
E_{\vp}^{(s)}(\Phi(t))\lesssim  \P( \|\vp\|_{L_t^\infty \dot{H}^s}, \|\psi\|_{L_t^\infty \dot{H}^s}) \left[ E_{\vp}^{(s)}(\Phi(0))+\|\Phi\|_{L^\infty_t \dot{H}^s}\|\Phi\|_{L^\infty_t\dot{H}^2}\|L\psi\|_{L^\infty_t \dot{H}^{s+1}} \right].
\end{equation}
The higher-order derivative term $\|L\psi\|_{L^\infty_t \dot{H}^{s+1}}$, which obstructs Lipschitz continuity on $\dot{H}^s$, is compensated by the lower-order derivative factor $\|\Phi\|_{L^\infty_t\dot{H}^2}$, and we treat it by approximating $\dot{H}^s$-solutions by smooth solutions.

Given $f\in L^2$ and $N\in \N$, we let $f_N = J_Nf$ where the projection $J_N$ is defined in \eqref{defJN}.
If $f\in \dot{H}^s$, with $s\ge 2$, then $f_N \to f$ in $\dot{H}^s$ as $N\to \infty$, and
\begin{equation}
\|f_N - f\|_{\dot{H}^2} \lesssim \frac{1}{N^{s-2}} \|f\|_{\dot{H}^s},\qquad
\|f_N\|_{\dot{H}^{s+1+\delta}} \lesssim {N^{1+\delta}} \|f\|_{\dot{H}^s}.
\label{approx_est}
\end{equation}

Consider initial data $\vp^n_0, \vp_0 \in \dot{H}^s$ such that  $\vp^n_0\to \vp_0$ in $\dot{H}^s$ as $n\to \infty$, and let $\vp^n, \vp\in C([0,T]; \dot{H}^s)$ denote the corresponding solutions. We approximate the initial data by
$\vp_{0,N}^n$, $\vp_{0,N}$ and let  $\vp_N^n$, $\vp_N$ denote the corresponding solutions. Then
\begin{equation}\label{triangle}
\|\vp^n-\vp\|_{\dot{H}^s}\leq \|\vp^n-\vp^n_N\|_{\dot{H}^s}+ \|\vp^n_N-\vp_N\|_{\dot{H}^s}+\|\vp_N-\vp\|_{\dot{H}^s}.
\end{equation}
Using \eqref{ContD} and the fact that $\|Lf\|_{L^2}\lesssim \|f\|_{\dot{H}^{\delta}}$, we get that
\begin{align*}
&\|\vp-\vp_N\|_{\dot{H}^s}^2
 \lesssim\P( \|\vp\|_{L_t^\infty \dot{H}^s}, \|\vp_N\|_{L_t^\infty \dot{H}^s})
\left[\|\vp_{0}-\vp_{0,N}\|_{\dot{H}^s}^2 +\|\vp-\vp_N\|_{L_t^\infty \dot{H}^s}\|\vp-\vp_N\|_{L_t^\infty \dot{H}^2}\|\vp_N\|_{L_t^\infty\dot{H}^{s+1+\delta}} \right],
\end{align*}
with a similar estimate for $\|\vp^n-\vp^n_N\|_{\dot{H}^s}^2$.
The Lipschitz continuity \eqref{stability} and the approximation estimates \eqref{approx_est} give
\begin{align*}
\|\vp-\vp_N\|_{L_t^\infty \dot{H}^2}\|\vp_N\|_{L_t^\infty\dot{H}^{s+1+\delta}}
\lesssim \|\vp_0-\vp_{0,N}\|_{\dot{H}^2}\|\vp_{0,N}\|_{\dot{H}^{s+1+\delta}}
\lesssim \frac{1}{N^{s-3-\delta}} \|\vp_0\|_{\dot{H}^s}^2.
\end{align*}
Hence, since $s > 7/2$, we have for each $n\in \N$ that
\begin{equation}
\|\vp-\vp_N\|_{L_t^\infty \dot{H}^s} + \|\vp^n-\vp^n_N\|_{L_t^\infty \dot{H}^s}
\to 0\qquad \text{as $N\to \infty$}.
\label{est2}
\end{equation}
In addition, using \eqref{ContD}, we get that
\begin{align*}
&\|\vp^n_N-\vp_N\|_{\dot{H}^s}^2
\\
&\quad \lesssim  \P( \|\vp^n_N\|_{L_t^\infty \dot{H}^s}, \|\vp_N\|_{L_t^\infty \dot{H}^s})
\left[\|\vp^n_{0,N}-\vp_{0,N}\|_{\dot{H}^s}^2 +\|\vp^n_N-\vp_N\|_{L_t^\infty \dot{H}^s}\|\vp^n_N-\vp_N\|_{L_t^\infty \dot{H}^2}\|\vp_N\|_{L_t^\infty\dot{H}^{s+1+\delta}} \right].
\end{align*}
Since $\vp^n_N \to \vp_N$ as $n\to \infty$, equation \eqref{stability}, with $r=2$, then implies that for each $N\in \N$, we have
\begin{equation}
\|\vp^n_N-\vp_N\|_{L_t^\infty\dot{H}^s} \to 0\qquad   \text{as $n\to \infty$}.
\label{est1}
\end{equation}
It follows from \eqref{triangle}--\eqref{est1} that
$\|\vp^n-\vp\|_{L_t^\infty \dot{H}^s}\to 0$ as $n\to \infty$, which proves
that the solution map $U$ is continuous on $\dot{H}^s$.

\end{document}